\documentclass[10pt]{amsart}
\pdfoutput=1
\usepackage{latexsym}
\usepackage{amssymb}
\usepackage{amscd}
\usepackage{epsf}
\usepackage{amsmath,amssymb,graphicx,amsthm,mathrsfs,verbatim}
\usepackage{pstricks}
\usepackage[colorlinks=true,pagebackref]{hyperref}

\begin{document}

\theoremstyle{plain}
\newtheorem{Thm}{Theorem}[section]
\newtheorem{TitleThm}[Thm]{}
\newtheorem{Corollary}[Thm]{Corollary}
\newtheorem{Proposition}[Thm]{Proposition}
\newtheorem{Lemma}[Thm]{Lemma}
\newtheorem{Conjecture}[Thm]{Conjecture}
\theoremstyle{definition}
\newtheorem{Definition}[Thm]{Definition}
\theoremstyle{definition}
\newtheorem{Example}[Thm]{Example}
\newtheorem{TitleExample}[Thm]{}
\newtheorem{Remark}[Thm]{Remark}
\newtheorem{SimpRemark}{Remark}
\renewcommand{\theSimpRemark}{}

\numberwithin{equation}{section}

\newcommand{\C}{{\mathbb C}}
\newcommand{\R}{{\mathbb R}}
\newcommand{\Z}{{\mathbb Z}}

\newcommand{\flushpar}{\par \noindent}

\newcommand{\proj}{{\rm proj}}
\newcommand{\coker}{{\rm coker}\,}
\newcommand{\rank}{{\rm rank}\,}
\newcommand{\supp}{{\rm supp}\,}
\newcommand{\codim}{{\operatorname{codim}}}
\newcommand{\sing}{{\operatorname{sing}}}
\newcommand{\Tor}{{\operatorname{Tor}}}
\newcommand{\wt}{{\operatorname{wt}}}
\newcommand{\dlog}{{\operatorname{Derlog}}}
\newcommand{\Olog}[2]{\Omega^{#1}(\text{log}#2)}
\newcommand{\produnion}{\cup \negmedspace \negmedspace 
\negmedspace\negmedspace {\scriptstyle \times}}
\newcommand{\pd}[2]{\dfrac{\partial#1}{\partial#2}}

\def \ba {\mathbf {a}}
\def \bb {\mathbf {b}}
\def \bc {\mathbf {c}}
\def \bone {\boldsymbol {1}}
\def \bh {\mathbf {h}}
\def \bk {\mathbf {k}}
\def \bm {\mathbf {m}}
\def \bn {\mathbf {n}}
\def \bt {\mathbf {t}}
\def \bu {\mathbf {u}}
\def \bv {\mathbf {v}}
\def \bx {\mathbf {x}}
\def \bw {\mathbf {w}}
\def \b1 {\mathbf {1}}
\def \bga {\boldsymbol \alpha}
\def \bgb {\boldsymbol \beta}
\def \bgg {\boldsymbol \gamma}
\def \bgw {\boldsymbol \omega}

\def \itc {\text{\it c}}
\def \ith {\text{\it h}}
\def \iti {\text{\it i}}
\def \itj {\text{\it j}}
\def \itm {\text{\it m}}
\def \itM {\text{\it M}} 
\def \itn {\text{\it n}}
\def \ithn {\text{\it hn}}
\def \itt {\text{\it t}}

\def \cA {\mathcal{A}}
\def \cB {\mathcal{B}}
\def \cC {\mathcal{C}}
\def \cD {\mathcal{D}}
\def \cE {\mathcal{E}}
\def \cF {\mathcal{F}}
\def \cG {\mathcal{G}}
\def \cH {\mathcal{H}}
\def \cK {\mathcal{K}}
\def \cL {\mathcal{L}}
\def \cM {\mathcal{M}}
\def \cN {\mathcal{N}}
\def \cO {\mathcal{O}}
\def \cP {\mathcal{P}}
\def \cS {\mathcal{S}}
\def \cT {\mathcal{T}}
\def \cU {\mathcal{U}}
\def \cV {\mathcal{V}}
\def \cW {\mathcal{W}}
\def \cX {\mathcal{X}}
\def \cY {\mathcal{Y}}
\def \cZ {\mathcal{Z}}

\def \ga {\alpha}
\def \gb {\beta}
\def \gg {\gamma}
\def \gd {\delta}
\def \ge {\epsilon}
\def \gevar {\varepsilon}
\def \gk {\kappa}
\def \gl {\lambda}
\def \gs {\sigma}
\def \gt {\tau}
\def \gw {\omega}
\def \gz {\zeta}
\def \gG {\Gamma}
\def \gD {\Delta}
\def \gL {\Lambda}
\def \gS {\Sigma}
\def \gW {\Omega}

\def \dim {{\rm dim}\,}
\def \mod {{\rm mod}\;}

\newcommand{\ds}{\displaystyle}
\newcommand{\vf}{\vspace{\fill}}
\newcommand{\vect}[1]{{\bf{#1}}}
\def\R{\mathbb R}
\def\C{\mathbb C}
\def\N{\mathbb N}
\def\Sym{\mathrm{Sym}}
\def\Sk{\mathrm{Sk}}
\def\GL{\mathrm{GL}}
\def\Diff{\mathrm{Diff}}
\def\id{\mathrm{id}}
\def\Pf{\mathrm{Pf}}
\def\sll{\mathfrak{sl}}
\def\g{\mathfrak{g}}
\def\h{\mathfrak{h}}
\def\k{\mathfrak{k}}
\def\t{\mathfrak{t}}
\def\OcN{\mathscr{O}_{\C^N}}
\def\Ocn{\mathscr{O}_{\C^n}}
\def\Ocm{\mathscr{O}_{\C^m}}
\def\Ocnz{\mathscr{O}_{\C^n,0}}
\def\E{\mathscr{E}}
\def\dimc{{\dim_{\C}}}
\def\dimC{{\dim_{\C}}}
\def\Derlog{\mathrm{Derlog}\,}
\def\expdeg{\mathrm{exp\,deg}\,}

\title[Solvable Groups and Free Divisors]
{Solvable Group Representations and Free Divisors whose Complements are 
$K(\pi, 1)$\lq s}
\author[James Damon and Brian Pike ]{James Damon$^1$ and Brian Pike} 
\thanks{(1) Partially supported by the National Science Foundation grant 
DMS-0706941}
\address{Department of Mathematics \\
University of North Carolina \\
Chapel Hill, NC 27599-3250  \\
USA}



\begin{abstract}
We apply previous results on the representations of solvable linear algebraic 
groups to construct a new class of free divisors whose complements are 
$K(\pi, 1)$\rq s.  These free divisors arise as the exceptional orbit varieties 
for a special class of \lq\lq block representations\rq\rq and have the 
structure of determinantal arrangements. \par  
Among these are the free divisors defined by conditions for the (modified) 
Cholesky-type factorizations of matrices, which contain the determinantal 
varieties of singular matrices of various types as components.  These 
complements are proven to be homotopy tori, as are the Milnor fibers of 
these free divisors.  The generators for the complex cohomology of each are 
given in terms of forms defined using the basic relative invariants of the 
group representation.
\end{abstract}

\maketitle
\vskip-.5\baselineskip
\centerline{{\it Department of Mathematics}} 
\centerline{{\it  University of North Carolina at Chapel Hill,}} 
\centerline{{\it Chapel Hill, NC 27599-3250, USA}}

\section*{Introduction}  
\label{S:sec0} 
A classical result of Arnold  and Brieskorn \cite{Bk}, \cite{Bk2} states that 
the complement of the discriminant of the versal unfolding of a simple 
hypersurface singularity is a $K(\pi, 1)$.  Deligne \cite{Dg} showed this 
result could be placed in the general framework by proving that the 
complement of an arrangement of reflecting hyperplanes for a Coxeter 
group is again a $K(\pi, 1)$ (and more generally for simplicial arrangements).  
A discriminant complement for a simple hypersurface singularity can be 
obtained as the quotient of the complement of such a hyperplane 
arrangement by the free action of a finite group, and hence is again is a 
$K(\pi, 1)$.  \par 
What the discriminants and Coxeter hyperplane arrangements have in 
common is that they are free divisors.  This notion was introduced by Saito 
\cite{Sa}, motivated by his discovery that the discriminants for the versal 
unfoldings of isolated hypersurface singularities are always free divisors.  By 
contrast, Kn\"{o}rrer \cite{Ko} found an isolated complete intersection 
singularity, for which the complement of the discriminant of the versal 
unfolding is not a $K(\pi, 1)$ (although it is again a free divisor by a result of 
Looijenga \cite{L}).  \par
This leads to an intriguing question about when a free divisor has a 
complement which is a $K(\pi, 1)$.  This remains unsettled for the 
discriminants of versal unfoldings of isolated hypersurface singularities; this 
is the classical \lq\lq $K(\pi, 1)$-Problem\rq\rq.  Also, for hyperplane 
arrangements, there are other families such as arrangements arising from 
Shephard groups which by Orlik-Solomon \cite{OS} satisfy both properties; 
however, it remains open whether the conjecture of Saito is true that every 
free arrangement has complement which is a $K(\pi, 1)$.  A survey of these 
results on arrangements can be found in the book of Orlik-Terao \cite{OT}.  
Except for isolated curve singularities in $\C^2$ (and the total space for 
their equisingular deformations), there are no other known examples of free 
divisors whose complements are $K(\pi, 1)$\rq s.  While neither $K(\pi, 1)$-
problem has been settled, numerous other classes of free divisors have been 
discovered so this question continues to arise in new contexts.  \par
In this paper, we define a large class of free divisors whose complements are 
$K(\pi, 1)$\rq s by using the results obtained in \cite{DP1}.  These free 
divisors are \lq\lq determinantal arrangements\rq\rq, which are analogous 
to hyperplane arrangements except that we replace a configuration of 
hyperplanes by a configuration of determinantal varieties (and the defining 
equation is a product of determinants rather than a product of linear 
factors).  \par
These varieties arise as the \lq\lq exceptional orbit varieties\rq\rq\, for 
representations of solvable linear algebraic groups which are \lq\lq Block 
Representations\rq\rq\, in the sense of \cite{DP1}; and in Theorem 
\ref{BlkKpione} we show that their complements are always $K(\pi, 1)$\rq s, 
where $\pi$ is the extension of a finitely generated free abelian group by a 
finite group.  More generally we show that for a weaker notion of \lq\lq  
nonreduced Block Representation\rq\rq, the exceptional orbit varieties are 
weaker free* divisors; however their complements are still $K(\pi, 1)$\rq s.   
From this we deduce in Theorem \ref{MilFibKpione} that the Milnor fibers of 
these exceptional orbit varieties 
are again $K(\pi, 1)$\rq s. \par
We exhibit in Theorem \ref{CholFacExam} a number of families of such \lq\lq 
determinantal arrangements\rq\rq in spaces of symmetric, 
skew-symmetric and general square matrices and 
$(m-1) \times m$--matrices which are free divisors (or free* divisors) with 
complements $K(\pi, 1)$\rq s.  We note that the individual determinant 
varieties in these spaces are neither free divisors nor are their complements 
$K(\pi, 1)$\rq s.  However, the determinant variety can be placed in a larger 
geometric configuration of determinantal varieties which together form a 
free divisor whose complement is a $K(\pi, 1)$.  \par
For these results we use special representations of solvable algebraic groups 
involved in various forms of Cholesky--type factorizations or modified 
Cholesky--type factorizations for symmetric, skew-symmetric and general 
square complex matrices, and $m \times (m+1)$ general matrices.  We 
describe these factorizations in \S \ref{S:sec1}.  We go on to specifically 
show in these cases $\pi \simeq \Z^k$ where $k$ is the rank of the 
corresponding solvable groups (Theorem \ref{CholFacExam}), so the 
complements are homotopy equivalent to 
$k$--tori.  We further deduce that the Milnor fibers for these cases are 
homotopy equivalent to $(k-1)$--tori.  Furthermore, in Theorem 
\ref{CohomComplFibKZm} and Corollary \ref{CorCholCohGM}, we are able to 
find explicit generators for the complex cohomology of the complement and 
of the Minor fibers using forms obtained from the basic relative invariants 
for the group actions, using results from the theory of prehomogeneous 
spaces due to Sato-Kimura \cite{SK}.  We deduce that the Gauss-Manin 
systems for these determinantal arrangements are trivial.  The simple form 
of these results contrasts with the more difficult situation of linear free 
divisors for reductive groups considered by Granger, Mond et al \cite{GMNS} 
\par 
These determinantal arrangements are also used in \cite{DP2} for 
determining the vanishing topology of more general matrix singularities 
based on the various determinantal varieties. \par
The authors are especially grateful to Shrawan Kumar for his comments and 
references.

\section{Cholesky Factorizations, Modified Cholesky Factorizations, and 
Solvable Group Representations}  
\label{S:sec1} 

In this section we begin by explaining the interest in determinantal 
arrangements which arises from various forms of Cholesky factorization. 
 Traditionally, it is well--known that certain matrices can be put in normal 
forms after multiplication by appropriate matrices.  The basic example is for 
symmetric matrices, where a nonsingular symmetric matrix $A$ can be 
diagonalized by composing it with an appropriate invertible matrix $B$ to 
obtain $B\cdot A \cdot B^T$.  The choice of $B$ is highly nonunique.  For 
real matrices, Cholesky factorization gives a unique choice for $B$ provided 
$A$ satisfies certain determinantal conditions.  \par 
 More generally, by \lq\lq Cholesky factorization\rq\rq\, we mean a general 
collection of results for factoring real matrices into products of upper and 
lower triangular matrices.  These factorizations are used to simplify the 
solution of certain problems in applied linear algebra.  \par
We recall the three fundamental cases (see \cite{Dm} and \cite{BBW}).  For 
them, we let $A = (a_{i j})$ denote an $m \times m$ real matrix which may 
be symmetric, general, or skew-symmetric.  We let $A^{(k)}$ denote the $k 
\times k$ upper left-hand corner submatrix. \par

\begin{Thm}[Forms of Cholesky--Type Factorization] \hfill
\label{CholFac} 
\begin{enumerate}
\item  {\em Classical Cholesky factorization: } If $A$ is a 
positive--definite symmetric matrix with $\det (A^{(k)}) \neq 0$  for $k = 1, 
\dots, m$, then there exists a unique lower triangular matrix with positive 
diagonal entries $B$ so that $A = B \cdot B^T$.
\item {\em Classical LU decomposition: } If $A$ is a general matrix with 
$\det (A^{(k)}) \neq 0$  for $k = 1, \dots, m$, then there exists a unique 
lower triangular matrix $B$ and upper triangular matrix $C$ with diagonal 
entries $= 1$ so that $A = B \cdot C$. 
\item  {\em Skew-symmetric Cholesky factorization (see e.g. \cite{BBW}): } 
If $A$ is a skew-symmetric matrix for $m  = 2\ell$ with $\det (A^{(2k)}) 
\neq 0$  for $k = 1, \dots, \ell$, then there exists a unique lower block 
triangular matrix  $B$ with $2 \times 2$--diagonal blocks of the form a) in 
(\ref{Eqn1.1}) so that $A = B \cdot J \cdot B^T$, for $J$ the block diagonal 
$2\ell \times 2\ell$ skew-symmetric matrix with 
$2 \times 2$--diagonal blocks of the form b) in (\ref{Eqn1.1}).  For $m = 
2\ell + 1$, then there is again a unique factorization except now $B$ has an 
additional entry of $1$ in the last diagonal position, and $J$ is replaced by 
$J^{\prime}$ which has $J$ as the upper left corner $2\ell \times 2\ell$ 
submatrix, with remaining entries $ = 0$.  
\begin{equation}
\label{Eqn1.1}
a) \qquad \begin{pmatrix}
r &0 \\ 0  &  \pm r 
\end{pmatrix}  \quad , r > 0\qquad \makebox{ and } \quad b) \qquad  \,\, 
\begin{pmatrix}
0 & -1 \\ 1  &  0 
\end{pmatrix}
\end{equation}
\end{enumerate}
\end{Thm}
\par
We note that in each case the polynomial in the entries of $A$, $\prod \det 
(A^{(k)})$ over $1 \leq k \leq m$ (with $k$ even in the skew-symmetric 
case), defines a real variety off which there is the appropriate Cholesky 
factorization defined.  This real variety is defined on a space of real 
matrices and can be viewed as a real determinantal arrangement formed 
from the real varieties defined by the individual $\det (A^{(k)})$.  We turn to 
the corresponding complex situation and identify such varieties as examples 
of determinantal varieties which arise as \lq\lq exceptional orbit 
varieties\rq\rq of solvable group actions.  This perspective leads to a more 
general understanding of the determinantal varieties associated to Cholesky 
factorization.  \par

\subsection*{Cholesky Factorizations and Determinantal Arrangements}
We begin with the notion of determinantal varieties and determinantal 
arrangements on a complex vector space $V$.  
\begin{Definition}
A variety $\cV \subset V$ is a {\it determinantal variety} if $\cV$ has a 
defining equation $p = \det(B)$ where $B = (b_{i, j})$ is a square matrix 
whose entries are linear functions on $V$.  Then, $X \subset V$ is a {\it 
determinantal arrangement} if $X$ has defining equation $p = \prod p_i$ 
where each $p_i$ is a defining equation for a determinantal variety $\cV_i$. 
Then, $X = \cup_i \cV_i$
\end{Definition}
In the simplest case where the determinants are $1 \times 1$ determinants, 
then we obtain a central hyperplane arrangement. \par
\begin{Remark}
In the definitions of determinantal variety and determinantal arrangements 
we do not require that the defining equations be reduced.  In fact, even for 
certain of the Cholesky--type factorizations above this need not be true.
\end{Remark}
 \par
We now consider the spaces of $m \times m$ complex matrices which will 
either be symmetric, general, or skew-symmetric (with $m$ even).  In the 
complex case there are the following analogues of Cholesky factorization 
(see \cite{DP1} and \cite{P}).
\begin{Thm}[Complex Cholesky--Type Factorization] \hfill
\label{ComplCholFac} 
\begin{enumerate}
\item  {\em Complex Cholesky factorization: } If $A$ is a 
complex symmetric matrix with $\det (A^{(k)}) \neq 0$  for $k = 1, \dots, 
m$, then there exists a lower triangular matrix $B$, which is unique up to 
multiplication by a diagonal matrix with diagonal entries $\pm 1$, so that $A 
= B \cdot B^T$. 
\item {\em Complex LU decomposition: } If $A$ is a general complex matrix 
with $\det (A^{(k)}) \neq 0$  for $k = 1, \dots, m$, then there exists a 
unique lower triangular matrix $B$ and a unique upper triangular matrix $C$ 
which has diagonal entries $= 1$ so that $A = B \cdot C$. 
\item  {\em Complex Skew-symmetric Cholesky factorization : } If $A$ is a 
skew-symmetric matrix for $m  = 2\ell$ with $\det (A^{(2k)}) \neq 0$  for 
$k = 1, \dots, \ell$, then there exists a lower block triangular matrix $B$ 
having the same form as in (3) of Theorem \ref{CholFac} but with complex 
entries of same signs in each $2 \times 2$ diagonal block a) of 
(\ref{Eqn1.1}) (i.e. $ = r\cdot I$), so that $A = B \cdot J \cdot B^T$, for $J$ 
the $2\ell \times 2\ell$ skew-symmetric matrix as in (3) of Theorem 
\ref{CholFac}.  Then, $B$ is unique up to multiplication by block diagonal 
matrices with a $2 \times 2$ diagonal blocks $= \pm I$.  There is also a 
factorization for the case $m  = 2\ell + 1$ analogous to that in (3) in 
Theorem \ref{CholFac}, with again with complex entries of same signs in 
each $2 \times 2$ diagonal block. 
\end{enumerate} 
\end{Thm}
\par
The polynomials $\prod \det (A^{(k)})$ over $1 \leq k \leq m$ (with $k$ even 
in the skew-symmetric case), define varieties which are determinantal 
arrangements.  However, these varieties have differing properties when 
viewed from the perspective of their being free divisors.  While they are free 
divisors in the symmetric case, they are a weaker form of free* divisor (see 
\cite{D}) for the general and skew-symmetric cases.  The stronger 
properties of free divisors discovered in \cite{DM} led to a search for a 
modification of the notion of Cholesky factorization for general $m \times 
m$ matrices.  This further extends to the space of $(m-1) \times m$ 
general matrices.  In each case there is a modified form of Cholesky--type 
factorization (see \cite{DP1} and \cite{P}) which we consider next.  \par
For an $m \times m$ matrix $A$, we let $\hat A$ denote the $m \times (m-
1)$ matrix obtained by deleting the first column of $A$.  If instead $A$ is an 
$(m-1) \times m$ matrix, we let $\hat A$ denote the $(m-1) \times (m-1)$ 
matrix obtained by deleting the first column of $A$.  In either case, we let 
$\hat A^{(k)}$ denote the $k \times k$  upper left submatrix of $\hat A$, 
for $1 \leq k \leq m-1$.  Then a modified form of Cholesky factorization is 
given by the following (see\cite{DP1} and \cite{P}).

\begin{Thm}[Modified Cholesky--Type Factorization] \hfill
\label{ModCholFac} 
\begin{enumerate}
\item {\em Modified LU decomposition: } If $A$ is a general complex $m 
\times m$ matrix with $\det (A^{(k)}) \neq 0$  for $k = 1, \dots, m$ and 
$\det (\hat A^{(k)}) \neq 0$  for $k = 1, \dots, m-1$, then there exists a 
unique lower triangular matrix $B$ and a unique  upper triangular matrix $C$, 
which has first diagonal entry $= 1$, and remaining first row entries $= 0$ 
so that $A = B \cdot K \cdot C$, where $K$ has the form of a) in 
(\ref{Eqn1.2}).
\item  {\em Modified Cholesky factorization for $(m-1) \times m$ matrices: 
} If $A$ is an $(m-1) \times m$ complex matrix with $\det (A^{(k)}) \neq 0$  
for $k = 1, \dots, m - 1$, $\det (\hat A^{(k)}) \neq 0$  for $k = 1, \dots, m-
1$, then there exists a unique $(m-1) \times (m-1)$ lower triangular matrix 
$B$ and a unique $m \times m$ matrix $C$ having the same form as in (1), 
so that $A = B \cdot K^{\prime} \cdot C$, where $K^{\prime}$ has the form 
of b) in (\ref{Eqn1.2}).  
\end{enumerate}
\begin{equation}
\label{Eqn1.2}
a) \qquad \begin{pmatrix}
1 &1 & 0 & 0 & 0 \\ 0  & 1 & 1 & 0 & 0 \\  0  & 0 & \ddots & 1 & 0 \\
0  & 0 & 0 & 1 & 1 \\  0  & 0  & 0  & 0 & 1
\end{pmatrix}  \qquad \makebox{ and } \quad b) \qquad    \begin{pmatrix}
1 &1 & 0 & 0 & 0 & 0 \\ 0  & 1 & 1 & 0 & 0 & 0 \\  0  & 0 & \ddots & 1 & 0 
& 0 \\
0  & 0 & 0 & 1 & 1 & 0 \\  0  & 0  & 0 & 0  & 1 & 1
\end{pmatrix}
\end{equation}
\end{Thm}
\par
This factorization yields unexpected forms and even one for 
non--square matrices.  The corresponding determinantal arrangements are 
defined by 
$$  \prod \det (A^{(k)}) \cdot \prod \det (\hat A^{(j)}) \,\, = \,\, 0$$
 where the products are over $1 \leq k \leq m$, $1 \leq j \leq m - 1$ for 
case (1) and over $1 \leq  k \leq m - 1$ and $1 \leq j \leq m - 1$ for case 
(2).  In the next section, we explain how as a consequence of \cite{DP1} and 
\cite{P}, these determinantal arrangements are free divisors.  As we will 
explain, these are special cases of a general result which constructs such 
determinantal arrangements from representations of solvable linear 
algebraic groups.  In fact there are many other families of determinantal 
arrangements which similarly arise (see \cite{DP1}).  This representation will 
then allow us to explicitly describe the complements to the determinantal 
arrangements and give criteria that they are $K(\pi, 1)$\rq s.

\section{Block Representations for Solvable Groups}  
\label{S:sec2} 
\par
All of the examples of Cholesky--type factorization given in \S \ref{S:sec1} 
can be viewed as statements about open orbits for representations of 
solvable linear algebraic groups.  For example, for the case of symmetric 
matrices, there is the representation of the Borel subgroup of $m \times m$ 
lower triangular matrices $B_m$ acting on the space of $m \times m$ 
symmetric matrices $Sym_m$ given by
\begin{equation}
C \cdot S \quad = \quad C\, S \, C^T \qquad  \makebox{ for $C \in B_m$ 
and $S \in Sym_m$}. 
\end{equation}
\par
However, not all such representations have the desired properties.  We 
consider a special class of finite dimensional (complex) regular 
representations $\rho : G \to \GL(V)$ of solvable linear algebraic groups $G$ 
(throughout this paper the solvable groups will always be understood to be 
connected).  Such a representation will be called an {\em equidimensional 
representation} if $\dim G = \dim V$ and $\ker(\rho)$ is finite.  We will 
specifically be interested in the case where $G$ has an open orbit, which is 
then Zariski open.  We refer to the complement, which consists of the orbits 
of positive codimension, as the {\it exceptional orbit variety} $\cE \subset 
V$.  \par
Mond first observed that in this situation it may be possible to apply Saito\rq 
s criterion to conclude that $\cE$ is a free divisor.  This has led to a new 
class of \lq\lq linear free divisors\rq\rq.  The question is when does Saito\rq 
s criterion apply.  In the case of reductive groups $G$, Mond and Buchweitz 
\cite{BM} used quivers of finite type to discover a large collection of linear 
free divisors.  We consider instead the situation for solvable algebraic 
groups.  
\begin{Remark} 
We note that such representations with a Zariski open orbit were studied 
many years ago by Sato and Kimura who called them {\it prehomogeneous 
vector spaces} except they did not require the representations to be 
equidimensional.  Also, they studied them from the perspective of harmonic 
analysis (see \cite{SK} and \cite{Ki}).  
\end{Remark}
\par 
We consider for a representation  $\rho : G \to \GL(V)$ the natural 
commutative diagram of (Lie) group and Lie algebra homomorphisms (see 
\cite{DP1}). \par
\vspace{1ex}
\flushpar
{\it Exponential Diagram for a Representation} \par
\begin{equation}
\label{CD.assoc.vfs}
\begin{CD} 
{\g} @>{\tilde{\rho}}>>  { \gl(V) } @>{\tilde{\iti}}>>
{\itm \cdot \theta(V)}\\
@V{\exp}VV     @V{\exp}VV      @V{\exp}VV   \\
{G}  @>{\rho}>>  {\GL(V)}  @>{\iti}>>  {\Diff(V, 0)}  
\end{CD} 
\end{equation}
 Here $\Diff(V, 0)$ denotes the group of germs of diffeomorphisms; 
$\theta(V)$ denotes the germs of holomorphic vector fields on $V, 0$, but 
with Lie bracket the negative of the usual Lie bracket; and  $\tilde{\iti}(A)$ 
is the vector field which at $v \in V$ has the value $A\cdot v$.   \par
For an equidimensional representation, the composition $\tilde \iti \circ 
\tilde \rho$ is a Lie algebra isomorphism onto its image.  The image of a 
vector $v \in \g$ will be denoted by $\xi_v$ and called a {\it representation 
vector field} associated to $v$.  For a basis $\{ v_i, i = 1, \dots , N\}$ of 
$\g$, we obtain $N$ representation vector fields $\{ \xi_{v_i}\}$.  \par
For a basis $\{w_j, j = 1, \dots , N\}$ of $V$, we can represent 
\begin{equation}
\label{Eqn1.7}
         \xi_{v_i} \,\, =  \,\, \sum_{j = 1}^{n} a_{j, i} w_j   \qquad  i = 1, \dots 
, N  
\end{equation}
where $a_{i, j} \in \cO_{V, 0}$.  We refer to the matrix $A = (a_{i, j})$ as the 
{\it coefficient matrix}.  Its columns are the coefficient functions for the 
vector fields.  For an equidimensional representation with open orbit, the 
exceptional orbit variety is defined (possibly with nonreduced structure) by 
the determinant $\det(A)$, which we refer to as the {\it coefficient 
determinant}.  As Mond observed, by Saito\rq s Criterion, if the coefficient 
determinant is a reduced defining equation for $\cE$, then $\cE$ is a free 
divisor which is called a {\it linear free divisor}.  We shall use the Lie algebra 
structure for the case of solvable algebraic groups to obtain linear free 
divisors.  \par
There is a special class of representations of solvable algebraic groups which 
we introduced in \cite{DP1}.
\begin{Definition}
\label{Def4.4}
  An equidimensional representation $V$ of a connected linear algebraic 
group $G$ will be called a {\em block representation} if:
\begin{itemize}
\item[i)]  there exists a sequence of $G$-invariant subspaces  
$$   V = W_k \supset W_{k-1} \supset \cdots  \supset W_{1}  \supset 
W_{0} = (0);  $$
\item[ii)] for the induced representation $\rho_j : G \to GL(V/W_j)$, we let  
$K_j = \ker(\rho_j)$, then $\dim K_j = \dim W_j$ for all $j$ and the 
equidimensional action of $K_j/K_{j-1}$ on $W_j/W_{j-1}$ has a relatively 
open orbit for each $j$;  
\item[iii)]  the relative coefficient determinants $p_j$ for the 
representations of $K_j/K_{j-1}$ on $W_j/W_{j-1}$ are all reduced and 
relatively prime in $\cO_{V, 0}$. 
\end{itemize} 
\end{Definition}
\par
\begin{Remark}
If in the preceding definition, both i) and ii) hold, with the relative coefficient 
determinants non-zero but possibly nonreduced or not relatively prime in 
pairs, then we say that it is a {\em nonreduced block representation}. 
\end{Remark} \par
The two terms \lq\lq relative coefficient determinants\rq\rq and  \lq\lq 
relatively open orbits\rq\rq are explained in more detail in \cite{DP1}.  For 
our purposes here, we can briefly explain their meaning by considering the 
coefficient matrix.  We choose a basis for $V$ and $\g$ formed from bases 
for the successive $W_j/W_{j-1}$ and $\k_j/\k_{j-1}$, $j = k, k-1, \dots 1$, 
where $\k_j$ is the Lie algebra of $K_j$.  We obtain a block triangular matrix 
coefficient matrix for the corresponding representation vector fields.  \par
\vspace{1ex}
{\bf Block Triangular Form:}  \hfill \par
\begin{equation}
\label{matr4.5}
\begin{pmatrix}
D_k &0 & 0 & 0 & 0 \\ *  & D_{k-1} & 0 & 0 & 0 \\  *  & * & \ddots & 0 & 0 
\\
*  & * &* & \ddots & 0 \\  *  & *  & *  & * & D_1
\end{pmatrix}
\end{equation}
The \lq\lq relative coefficient determinants\rq\rq are $p_j = \det(D_j)$.  
These are polynomials defined on $V$ even though they are for the 
representation $W_j/W_{j-1}$.  Also, the condition for \lq\lq relative open 
orbits\rq\rq\, is equivalent to $p_j$ not being identically $0$ (on $V$).

\par
Then, block representations give rise to free divisors, and in the nonreduced 
case to free* divisors (see \cite{DP1} or \cite{P}).
\begin{Thm}
\label{BlkRepFrDiv} 
Let $\rho : G \to  \GL(V)$ be a block representation of a solvable linear 
algebraic group $G$, with relative coefficient determinants $p_j,\, j = 1, 
\dots, k$. Then, the \lq\lq exceptional orbit variety\rq\rq\ $\cE, 0 \subset 
V, 0$ is a linear free divisor with reduced defining equation $\prod_{j = 
1}^{k} p_j = 0$.  \par
If instead $\rho : G \to  \GL(V)$ is a nonreduced block representation, then 
$\cE, 0 \subset V, 0$ is a linear free* divisor and $\prod_{j = 1}^{k} p_j = 0$ 
is a nonreduced defining equation for $(\cE, 0)$.
\end{Thm}
\par
In \cite{DP1} it is shown that all of the determinantal arrangements arising 
from Cholesky-type factorizations in \S \ref{S:sec1} are in fact the 
exceptional orbit varieties for the equidimensional representations of 
appropriate solvable linear algebraic groups.    There is then the following 
consequence for these determinantal arrangements

\begin{Thm} \hfill
\label{Choldetarr} 
\begin{itemize}
\item[i)]  The determinantal arrangements arising from the cases of 
Cholesky--type factorization for complex symmetric matrices and  modified 
Cholesky--type factorization for complex general $m \times m$ and $(m-1) 
\times m$ matrices are exceptional orbit varieties for the block 
representations of the corresponding solvable algebraic groups. As such 
they are free divisors.
\item[ii)]  The determinantal arrangements arising from the 
Cholesky--type factorization for complex general $m \times m$ matrices, 
and $m \times m$ skew-symmetric matrices (m even) are exceptional orbit 
varieties for the nonreduced block representations of the corresponding 
solvable algebraic groups.  As such they are free* divisors.
\end{itemize}
\end{Thm}
\begin{proof}
We list in Table \ref{table2.0}, each type of complex (modified) 
Cholesky--type factorization, the space of complex matrices, and the 
solvable group and representation which define the factorization.  For this 
table we use the notation that the spaces of complex $m \times m$ 
matrices are denoted by:  $Sym_m$ for symmetric matrices, $M_{m, m}$ 
for general $m \times m$ matrices, and $Sk_m$ for skew--symmetric 
matrices.  We also let  $M_{m-1, m}$ denote the space of complex ${m-1} 
\times m$ matrices.  For the groups we use the notation:  $B_m$ for the 
Borel group of $m \times m$ lower triangular matrices; $N_m$ for the 
nilpotent group of $m \times m$ upper triangular matrices with $1$\rq s on 
the diagonal; for $m = 2\ell$, $D_m$ denotes the group of lower block 
triangular matrices, with $2 \times 2$ diagonal blocks of the form a) in 
(\ref{Eqn1.1}) with complex entries of the same sign $r \neq 0$ (i.e. $r\cdot 
I$); while for $m = 2\ell +1$, $D_m$ denotes the group of lower block 
triangular matrices with the first $\ell\,$ $(2 \times 2)$-diagonal blocks as 
above and the last diagonal element $= 1$; and $C_m$ is subgroup of the $m 
\times m$ upper triangular matrices with $1$ in the first entry and other 
entries in the first row $= 0$. \par
These are each either block or nonreduced block representations, as is 
shown in \cite{DP1}, so that Theorem \ref{BlkRepFrDiv} applies.
\end{proof}
\begin{table}
\begin{tabular}{|l|c|c|l|}
\hline
Cholesky--type  & Matrix Space & Solvable Group & Representation \\
Factorization  &     &      &     \\
\hline
symmetric matrices & $Sym_m$ & $B_m$  &  $B\cdot A = B\, A\, B^T$\\
general  matrices  &$M_{m, m}$ & $B_m \times N_m$  &  $(B, C)\cdot A = 
B\, A\, C^{-1}$\\
skew-symmetric &  $Sk_m$  & $D_m$  &  $B\cdot A = B\, A\, B^T$  \\
\hline \hline
Modified Cholesky  &     &    &     \\
--type Factorization  &     &      &     \\
\hline
general $m \times m$ & $M_{m, m}$  &  $B_m \times C_m$   &   $(B, 
C)\cdot A = B\, A\, C^{-1}$   \\
general $(m-1) \times m$  & $M_{m-1, m}$ & $B_{m-1} \times C_m$   &   
$(B, C)\cdot A = B\, A\, C^{-1}$     \\
\hline

\end{tabular}
\vspace*{.2cm}
\caption{\label{table2.0} Solvable groups and (nonreduced) Block 
representations for (modified) Cholesky--type Factorization.}
\end{table}

\begin{Remark}
For the case of skew-symmetric matrices, we have not found a modified 
form of Cholesky factorization for which the resulting determinantal 
arrangement is a free divisor.  However, by extending the results to 
representations of nonlinear solvable infinite dimensional Lie algebras, we 
have found a free divisor which is the analogue of the exceptional orbit 
variety (again see \cite{DP1} and \cite{P}).
\end{Remark}
\section{Complements of Exceptional Orbit Varieties for Block 
Representations of Solvable Groups}  
\label{S:sec3} 
We next see that for block (or nonreduced block) representations, not only 
are the exceptional orbit varieties free divisors (resp. free* divisors), but 
they also share the additional property of having a complement which is a 
$K(\pi, 1)$.  \par
\begin{Thm}
\label{BlkKpione} 
Let $\rho : G \to  \GL(V)$ be a block representation of a solvable linear 
algebraic group $G$ whose rank is $m$.  Then, the complement of the 
exceptional orbit variety, $V \backslash \cE$, is a $K(\pi, 1)$ where $\pi$ is 
isomorphic to an extension of $\Z^m$ by the finite isotropy subgroup for a 
generic $v_0 \in V$.  \par
If instead $\rho : G \to  \GL(V)$ is a nonreduced block representation, then 
although $\cE, 0 \subset V, 0$ is only a linear free* divisor, the complement 
$V \backslash \cE$ is still a $K(\pi, 1)$ with $\pi$ as above.  
\end{Thm}
\begin{proof}
Let $\cU$ denote the Zariski open orbit of $G$.  We choose $v_0 \in \cU$. 
The map $G \to \cU$ sending $g \mapsto g\cdot v_0$ is surjective, as is the 
corresponding derivative map.  By the equidimensionality, the isotropy 
subgroup $H \subset G$ for $v_0$ is a $0$-dimensional algebraic group, and 
hence finite.  By standard results for Lie groups, the induced mapping $G/H 
\to \cU$ sending $gH \mapsto g\cdot v_0$ is a diffeomorphism.  As $G$ is 
connected, $p: G \to G/H$ is a fiber bundle with finite fiber and connected 
total space; hence, it is a finite covering space.  \par
Also, by the structure theorem for connected solvable groups, $G$ is the 
extension of its maximal torus $(\C^*)^m$ by its unipotent radical $N$.  It is 
a standard result for algebraic groups that the nilpotent group $N$ is a 
Euclidean group, i.e. the underlying manifold is diffeomorphic to some $\C^k$ 
(for example, by Corollary 4.8 in \cite{Bo}, it is a subgroup of some upper 
triangular group in $SL_n (\C)$, and then Corollary 1.134 and Theorem 1.127  
of \cite{Kn} yield the result).  Hence, $G$ has the homotopy type of its 
maximal torus, which is a $K( \Z^m, 1)$, where $m = \rank(G)$.  \par 
Hence, a simple argument using the homotopy exact sequence for a fibration 
shows that $G/H$ is also a $K(\pi, 1)$ and by basic results on covering 
spaces, $H \simeq \pi_1(G/H, \bar{e})/ p_*(\pi_1(G, e))$ (for $\bar{e} = 
e\cdot H$).  Thus, $G/H$ is a $K(\pi, 1)$ with $\pi$ isomorphic to the 
extension of $\Z^m$ by $H$.  
\end{proof}
As a consequence we are able to describe the Milnor fiber of the 
nonisolated hypersurface singularity $(\cE, 0)$.
\begin{Thm}
\label{MilFibKpione}  
Let $\rho : G  \to  \GL(V)$ be a (nonreduced) block representation of a 
solvable linear algebraic group $G$ whose rank is $m$, with $ \cE$ the 
exceptional orbit variety.  Then, the Milnor fiber of the nonisolated 
hypersurface singularity $ (\cE, 0)$ is a $K(\pi, 1)$, where $\pi$ satisfies 
\begin{equation}
\label{CD1}
\begin{CD} 
{0 } @>>> {\pi} @>>>  { \pi_1(V \backslash \cE) } @>>> {\Z}  @>>>  {0 }
\end{CD}   \\
\end{equation}
with $\pi_1(V \backslash \cE)$ isomorphic to an extension of $\Z^m$ by 
the finite isotropy subgroup for a generic $v_0 \in V$.  \par
\end{Thm} 
\begin{proof}
\par 
First, we observe that if $h$ is the reduced homogeneous defining equation 
for $\cE$, then $h : V \backslash \cE \to \C^*$ is a global fibration.  This 
follows using the $\C^*$--action.   If $h$ has degree  $d$, we can find an 
open neighborhood $U$ of $1$ in $\C^*$, invariant under inverses, so that 
the function $f(z) = z^d$ has a well-defined branch of the inverse $d$--th 
root function, which we denote by $\theta$.  For $w_0 \in \C^*$, there is a 
neighborhood $W$ of $w_0$ obtained by applying the $\C^*$--action to $U$, 
$z \mapsto  z\cdot w_0$. \par
Then,  a local trivialization is given by
\begin{align}
\label{aig3,1}
\psi : W \times h^{-1}(w_0) &\to h^{-1}(W) \\
(w, (z_1, \dots , z_n)) &\mapsto  \theta(z)\cdot (z_1, \dots , z_n) \notag
\end{align}
where $z = w/w_0$.  Then  by the homogeneity of $h$, 
$$h( \theta(z)\cdot  z_1, \dots , \theta(z)\cdot  z_n) \,\, =  \,\,  
\theta(z)^d h( z_1, \dots ,  z_n) \,\, = \,\, z\cdot w_0 \,\, = \,\, w $$
\par
As $\C^*$ is connected all fibers of $h : V \backslash \cE \to \C^*$ are 
diffeomorphic.  We next show that the Milnor fiber of $(\cE, 0)$ is 
diffeomorphic to a fiber of this fibration.   Given $\gevar > 0$ and $\gd > 0$ 
sufficiently small so that 
$h^{-1}(B_{\gd} \backslash \{ 0\}) \cap B_{\gevar}$ is the Milnor fibration 
of $(\cE, 0)$ and if $w \in B_{\gd}\backslash \{ 0\}$, then the fibers 
$h^{-1}(w)$ are tranverse to the $\gevar$--sphere $S^{2n-1}_{\gevar}$ 
about $0$.  
We further claim that by the $\C^*$--action, the fibers $h^{-1}(w)$, for $w 
\in B_{\gd}\backslash \{ 0\}$ are transverse to all spheres 
$S^{2n-1}_{R}$ for $R > \gevar$.  If $z \in h^{-1}(w)$ with $\| z\| = R  > 
\gevar$, we let $a = \frac{\gevar}{R}$ and $z^{\prime} = a\cdot z$.  Then, 
$$  h(z^{\prime})\,\,  = \,\, h(a\cdot z) \,\,  = \,\, a^d\, h(z)\,\,  = \,\, 
\left(\frac{\gevar}{R}\right)^d\cdot w = w^{\prime}\, . $$
Thus, $h(z^{\prime}) \in  B_{\gd}\backslash \{ 0\}$ and 
$\| z^{\prime} \| = a\cdot \| z \| = \gevar$.  Multiplication by $a$ sends 
$S^{2n-1}_{R}$ to $S^{2n-1}_{\gevar}$ and $h^{-1}(w)$ to 
$h^{-1}(w^{\prime})$.  Since $h^{-1}(w^{\prime})$ is transverse to 
$S^{2n-1}_{\gevar}$ at $z^{\prime}$, and transversality is preserved under 
diffeomorphisms, we conclude that $h^{-1}(w)$ is transverse to 
$S^{2n-1}_{R}$ at $z$.  \par  
Hence, on the fibers $X = h^{-1}(w) \backslash B_{\gevar}$, the function 
$g(z) = \| z\|$ has no critical points.   It then follows using Morse theory that 
$h^{-1}(w)$ is diffeomorphic to the Milnor fiber $h^{-1}(w) \cap B_{\gevar}$ 
as claimed.  \par
Finally, it is sufficient to show that a fiber $F$ of the fibration $h : V 
\backslash \cE \to \C^*$ is a $K(\pi, 1)$ with $\pi$ satisfying  (\ref{CD1}).  
As $F$ is diffeomorphic to the Milnor fiber of $(\cE, 0)$, we can at least 
conclude by e.g. the Kato-Matsumoto theorem that they are both $0$--
connected, i.e. path--connected (in the special case of $\dim V = 1$ it is 
trivially true).  Next, by the homotopy exact sequence for the fibration, we 
have   
\begin{equation}
\label{CD2}
\begin{CD} 
{\pi_{j+1} (\C^*)} @>>> {\pi_j(F) } @>>>  { \pi_j(V \backslash \cE) } @>>> 
{\pi_j(\C^*)}  @>>>  {\pi_{j-1}(F) }
\end{CD}   \\
\end{equation}
If $j >1$, then both $\pi_j(V \backslash \cE) = 0$, $\pi_{j+1}(\C^*) = 0$; 
hence, $\pi_j(F) = 0$ for $j > 1$. Thus, $F$ is a $K(\pi, 1)$.  Also, as $F$ is 
path--connected, then the long exact sequence (\ref{CD2}) with $j = 1$ 
yields (\ref{CD1}). 
\end{proof}
 \par 
As a corollary we have an important special case.
\begin{Corollary}
\label{MilFibKZm} 
Let $\rho : G \to  \GL(V)$ be a (nonreduced) block representation of a 
solvable linear algebraic group $G$ whose rank is $m$ so that the 
complement of the exceptional orbit variety $ \cE$ satisfies  
$\pi_1(V\backslash \cE) \simeq \Z^m$.  Then, $V\backslash \cE$ is 
homotopy equivalent to an $m$--torus, and the Milnor fiber of the 
nonisolated hypersurface singularity $(\cE, 0)$ is homotopy equivalent to an 
$(m - 1)$--torus.  
\end{Corollary}
\begin{proof}
By the hypothesis and Theorem \ref{BlkKpione}, $V\backslash \cE$ is a 
$K( \Z^m, 1)$ and hence is homotopy equivalent to the $m$--torus.  Second, 
by Theorem \ref{MilFibKpione}, the Milnor fiber is a $K(\pi, 1)$ where $\pi$ is 
a subgroup of $\Z^m$ with quotient $\Z$.  Thus, $\pi$ is a free abelian group 
and  by comparing ranks, $\pi \simeq \Z^{m-1}$.   Thus, the Milnor fiber is 
homotopy equivalent to an $(m-1)$--torus.
\end{proof}
\par
One example of the usefulness of these theorems is their general 
applicability to complements of determinantal arrangements in spaces of 
matrices corresponding to Cholesky or modified Cholesky--type 
factorizations, as well as to other (nonreduced) block representations given 
in \cite{DP1} and \cite{DP2}.
\subsection*{Determinantal Arrangements whose Complements are $K(\Z^k, 
1)$\rq~s}  \hfill
\par 
We return to the determinantal arrangements arising from Cholesky or 
modified Cholesky type factorizations that we considered in \S \ref{S:sec1}.  
We have the following general result for the topology of their complements.
\begin{Thm}
\label{CholFacExam}  
Each of the determinantal arrangements $\cE$ associated to the complex 
Cholesky--type factorizations in Theorem \ref{ComplCholFac} and the 
modified Cholesky--type factorizations in Theorem \ref{ModCholFac} 
have complements which are $K(\Z^k, 1)$\rq~s, where $k$ is the rank of the 
corresponding solvable group in Table \ref{table2.0}.  Hence, they are 
homotopy equivalent to $k$--tori; and the Milnor fibers of $(\cE, 0)$ are 
homotopy equivalent to $(k-1)$--tori.  \par
For the three families of complex symmetric matrices, and modified 
Cholesky factorizations for general complex $m \times m$ and 
$(m-1) \times m$ matrices the corresponding determinantal 
arrangements $\cE$ are free divisors with the preceding properties.
\end{Thm}
\begin{proof}
As stated in Theorem \ref{Choldetarr}, the corresponding determinantal 
arrangements for the complex Cholesky and modified Cholesky-type 
factorizations are the exceptional orbit varieties for the corresponding 
solvable linear algebraic groups given in Table \ref{table2.0}.  Then, Theorem 
\ref{BlkKpione} implies that the complements of the determinantal 
arrangements are $K(\pi, 1)$\rq s.  Once we have shown that in each case 
$\pi \simeq \Z^k$, where $k$ is the rank of the corresponding solvable 
group, it will follow by Corollary \ref{MilFibKZm} that the complement is 
homotopy equivalent to a $k$--torus and the Milnor fiber, to a $(k-1)$--torus.  
Furthermore, by Theorem \ref{Choldetarr} for the cases of complex 
symmetric matrices and the modified Cholesky factorizations, the 
corresponding exceptional orbit varieties are free divisors with the preceding 
properties.  
It remains to show in each case that $\pi_1(V\backslash \cE) \simeq  \Z^k$ 
where $k$ is the rank of the corresponding solvable group. 
\par
We first consider the determinantal arrangements for the complex 
Cholesky--type factorization (see Theorem \ref{ComplCholFac}).  In the case 
of $m \times m$ complex symmetric matrices, the isotropy group for the 
identity matrix $I$ is $H = (\Z/2\Z )^m$ consisting of diagonal matrices with 
entries $\pm 1$ as the diagonal entries.  We claim that the extension of 
$\Z^m$ by $H$ is again isomorphic to $\Z^m$.  
\par  
To see this, we consider for  a $1 \leq j \leq m$ a path $\gamma_j(t)$ from 
$[0, 1]$ to the Borel subgroup $B_m$ of lower triangle matrices.  It consists 
of diagonal matrices $\gamma_j(t) = B_t$ with entries $1$ in all positions 
except for  the $j$--th which is $e^{\pi i t}$.  Then $B_t\cdot B_t^T$ is 
diagonal with all diagonal entries $1$ except in the $j$--th position, where it is 
$e^{2\pi i t}$.  This is a closed path $\ga_j$ in the complement of the 
determinantal arrangement.  Thus, the corresponding path in the Borel 
subgroup $B_m$ is a lift of $\ga_j$.  Also, a lift of $\ga_j*\ga_j$ is the path 
$\gb_j$ in $B_m$ of diagonal matrices with $j$--th entry $e^{2\pi i t}$ which 
defines the $j$--th generator of $\Z^m$, the fundamental group for $B_m$, 
and hence $G$.  
\par  
Second,  the covering transformation $h_j$ of $G$ corresponding to $\ga_j$ 
is given by multiplication by the diagonal matrix $H_j$, whose 
$(i, i)$--entry is $1$ if $i\neq j$ and $-1$ if $i = j$.  These generate the group 
of covering transformations.
\par  
Third, because paths $\ga_i(t)$ and $\ga_j(t)$ with $i \neq j$ are 
in different diagonal positions, the path classes $\ga_j*\ga_i$ and 
$\ga_i*\ga_j$ are homotopic.  Hence, the classes in $\pi_1(V \backslash 
\cE, I)$ defined by  $\{ \ga_i(t) : 1 \leq i \leq m\}$ commute; they generate 
the group of covering transformations; and their squares generate $\pi_1(G, 
I)$.  Thus, they generate $\pi_1(V \backslash \cE, I)$, which is then a free 
abelian group generated by the $\ga_i$.  
\par  
Second, for general $m \times m$ complex matrices, by the uniqueness of 
the complex LU-decomposition, the isotropy group is the trivial group.  
Hence, $\pi \simeq \Z^m$.  
\par  
Third, for $m \times m$ skew-symmetric matrices with $m = 2\ell$, the 
isotropy subgroup $H$ of the matrix $J$ in (3) is $H \simeq  
(\Z/2\Z)^{\ell}$.  The generator of the $j$--th factor is given by the block 
diagonal matrix with $2 \times 2$ blocks which are the identity except for 
the $j$--th block which can be the $2 \times 2$ diagonal matrix 
$\pm I$.  An analogous argument as for the symmetric case shows that the 
extension of $\Z^{\ell}$ by $H$ is again isomorphic to $\Z^{\ell}$.  For $m = 
2\ell + 1$, a similar argument likewise shows that $H \simeq  
(\Z/2\Z)^{\ell}$ and the extension of $\Z^{\ell}$ by $H$ is again isomorphic 
to $\Z^{\ell}$. \par
Fourth, for both types of modified Cholesky-type factorization the 
factorization is unique.  Hence, in both cases the isotropy subgroups are 
trivial.  Hence, again $\pi \simeq \Z^k$, where $k$ is the rank of the 
corresponding solvable group in Table \ref{table2.0}. It is $k = 2m-1$ for the 
$m \times m$ general case  and $k = 2m -2$ for the $(m-1) \times m$ case.  
\end{proof}

\begin{Example}
\label{Exam3.1}
We illustate the preceding with the simplest examples.  We consider the 
lowest dimensional representations of each type.  The matrices in each 
space are given by
\begin{equation}
\label{Eqn3.1a}
a) \,\, \begin{pmatrix}
x  &  y \\
y  & z  
\end{pmatrix} \quad b) \,\, \begin{pmatrix}
x  &  y \\
z & w  
\end{pmatrix} \quad c)  \,\, \begin{pmatrix}
x  &  y  & z \\
u  & v  & w 
\end{pmatrix} \quad d)  \,\, \begin{pmatrix}
0 & x  &  y  & z \\
-x & 0  & u & v \\
-y & -u  & 0 & w \\
- z & -v  & -w  & 0
\end{pmatrix}
\end{equation}
\par Then, Table \ref{table3.0} lists the corresponding representation and 
the topological type of both the complement and  the Milnor fiber of the 
exceptional orbit varieties.  One point to observe is that the equations $xz - 
y^2$ on $\C^3$, $xw - yz$ on $\C^4$, and $xw - yv + zu$ on $\C^6$ define 
Morse singularities at $0$.  Their Milnor fibers are homotopy equivalent to 
respectively $S^2$, $S^3$, and $S^5$; and the complements are homotopy 
equivalent to bundles over $S^1$ with these respective fibers.  By adding a 
plane tangent to an element of each of the cones defined by the equations, 
the complements and Milnor fibers become homotopy tori.
\begin{table}
\begin{tabular}{|l|c|c|c|c|l|}
\hline
 Matrix Space & Group & Free/Free* & $\cE$  & $V\backslash \cE$  & 
Milnor\\
   &  &   &    &    &   Fiber  \\
\hline
Cholesky--type  &  &   &    &    & \\
 Factorization  &   &  &      &   &  \\
\hline
$Sym_2$ & $B_2$ & Free &  $x\,(xz - y^2)$ & $T^2$  & $S^1$\\
$M_{2, 2}$ & $B_2 \times N_2$ & Free* &  $x\,(xw - yz)$ & $T^2$  & 
$S^1$\\
$Sk_4$  & $D_4$  &  Free* &  $x\,(xw - yv + zu)$ & $T^2$  & $S^1$ \\
\hline \hline
Modified   &   &  &    &    & \\
Cholesky--type  &   &  &      &   &  \\
\hline
$M_{2, 2}$  &  $B_2 \times C_2$   &  Free  &  $xy\,(xw - yz)$ & $T^3$  & 
$T^2$  \\
$M_{2, 3}$ & $B_2 \times C_3$ & Free &  $xy\,(xv - yu)$ & $T^4$  & $T^3$ 
\\
     &   &    &  $\cdot (yw - zv)$    &   &  \\
$Sk_4$  & nonlinear  &  Free &  $xyu\, (yv - zu)$ &    &   \\
     &   &    &  $\cdot (xw - yv + zu)$    &   &  \\

\hline
\end{tabular}
\vspace*{.2cm}
\caption{\label{table3.0} Simplest examples of representations for 
(modified) Cholesky--type factorizations, with equations defining  exceptional 
orbit varieties $\cE$. Listed is the homotopy type of the complement 
$V\backslash \cE$ and the Milnor fiber of $\cE$.  Note that because the 
nonlinear group acting on $Sk_4$ is infinite dimensional, we cannot apply the 
preceding results to determine the topology of the complement nor the 
Milnor fiber.}
\end{table}
\end{Example}
\section{Generators for the Cohomology of the Milnor Fibers}
\label{Sec:4}\par
For the cases of the representations corresponding to both Cholesky-type 
and modified Cholesky-type factorizations, we will compute explicit 
generators for the cohomology algebras with complex coefficients of both 
the complement and the Milnor fiber of the exceptional orbit variety.  By  
Theorem \ref{CholFacExam}, it is enough to give a basis for $H^1(\cdot, 
\C)$ for each case.  
\par
We will use several facts concerning prehomogeneous spaces due to Sato 
and Kimura (see e. g. \cite{Ki} and \cite{SK}).  Prehomogeneous spaces are 
representations $V$ of complex algebraic Lie groups $G$ which have open 
orbits.  The exceptional orbit variety $\cE$ is again the complement of the 
open orbit.  By  Theorem 2.9 in \cite{SK}, the components of $\cE$ which 
are hypersurfaces have reduced defining equations $f_i$ which are the {\em 
basic relative invariants} (they generate the multiplicative group of all 
relative invariants).  $f_i$ is a relative invariant if there is a rational 
character $\chi_i$ of $G$ so that $f_i(g\cdot v) = \chi_i(g)\cdot f_i(v)$ for 
all $g \in G$ and $v \in V$.  The $1$--forms $\gw_i = \frac{df_i}{f_i}$ are 
defined on $V \backslash \cE$ and are the pull-backs of $\frac{dz}{z}$ via 
$f_i$.  Hence, they are closed and define 
$1$--dimensional cohomology classes in $H^1( V \backslash \cE , \C)$.  By 
pulling back the $\gw_i$ via the inclusion map of the Milnor fiber $F 
\hookrightarrow V \backslash \cE$ we obtain cohomology classes $\tilde 
\gw_i \in H^1(F, \C)$.  We have the following description of the cohomology 
algebras.  
\begin{Thm}
\label{CohomComplFibKZm} 
Let $\rho : G \to  \GL(V)$ be a (nonreduced) block representation of a 
solvable linear algebraic group $G$ whose rank is $k$.  Suppose the 
complement of the the exceptional orbit variety $\cE$  satisfies 
$\pi_1(V\backslash \cE) \simeq \Z^k$.  Then, \par
\begin{itemize}
\item[i)] there are $k$ basic relative invariants  $f_i$ and $H^1(V\backslash 
\cE, \C)$ has the basis $\gw_i$ for $i = 1, \dots k$.  Hence, the cohomology 
algebra $H^*(V\backslash \cE, \C)$ is the free exterior algebra on these 
generators;
\item[ii)] $H^1(F, \C)$ is generated by the $\{ \tilde \gw_i, i = 1, \dots , 
k\}$ with a single relation $\sum_{i = 1}^{k} \tilde \gw_i = 0$.  Hence,  
$H^*(F, \C)$ is the free exterior algebra on any subset of $k -1$ of the 
$\tilde \gw_i$.
\end{itemize}
\end{Thm}
\par 
Because of the explicit generators and relation for the degree $1$ 
cohomology of the Milnor fiber, we can draw the following concluson.
\begin{Corollary}  
\label{CorGauMan}
For a (nonreduced) block representation as in Theorem 
\ref{CohomComplFibKZm}, the Gauss-Manin connection for the exceptional 
orbit variety $(\cE, 0)$ is trivial.  
\end{Corollary}
\begin{proof}[Proof of Corollary \ref{CorGauMan}]
By Theorem \ref{CohomComplFibKZm}, the $\gw_i$ restrict to give global 
sections of the cohomology sheaf $\cH^1(U, \C) = H^1(h^{-1}(U), \C)$ for 
$U \subset \C^*$ (and  $h$ the reduced defining equation).  Hence, the 
Gauss--Manin connection for the fibration $h :V\backslash \cE \to \C^*$ is 
trivial for each of these elements, as it is for the single relation $\sum 
\gw_i$ .  Since their restrictions generate the cohomology of the fiber, the 
Gauss-Manin connection acts trivially on the entire cohomology.  As the 
inclusion of the Milnor fibration of the exceptional orbit variety into 
$V\backslash \cE$ is a homotopy equivalence of fibrations, the Gauss Manin 
connection is also trivial on the Milnor fiber.
\end{proof}
By Theorem \ref{CholFacExam}, this theorem applies to all of the 
representations corresponding to both Cholesky-type and modified Cholesky-
type factorizations.
\begin{Corollary}  
\label{CorCholCohGM}
For each representation $\rho : G \to GL(V)$ corresponding to a Cholesky or 
modified Cholesky type factorizations, the conclusions of Theorem 
\ref{CohomComplFibKZm} and Corollary \ref{CorGauMan} apply to the 
complement of the exceptional orbit variety $\cE$, and to the Milnor fiber 
and Gauss--Manin connection of $(\cE, 0)$.
\end{Corollary}
\begin{Example}
\label{Exam.Cohom}
\par
For the representation of $B_3$ on $Sym_3$, the exceptional orbit variety 
$\cE_3^{sy}$ is defined using coordinates for a generic matrix 
$$ A \,\, = \,\, \begin{pmatrix}
x & y & z \\
y  & w & u \\
z & u & v 
\end{pmatrix}  $$ \par
\vspace{2ex} 
 by  
$$  x\,(xw -y^2)\cdot\det (A) \,\, = \,\,0 \, .$$  \par
By Theorem \ref{CohomComplFibKZm}  and Corollary \ref{CorCholCohGM}, 
the complex cohomology of the complement is the exterior algebra 
$$  H^*(Sym_3 \backslash \cE_3^{sy}; \C) \,\, \simeq \,\, \gL^* \C< 
\frac{dx}{x}, \frac{d(xw -y^2)}{(xw -y^2)}, \frac{d(\det (A))}{\det (A)} >  
$$
In addition, the complex cohomology of the Milnor fiber of $\cE_3^{sy}$ is 
isomorphic to the exterior algebra on any two of the preceding generators.  
\end{Example}
\begin{proof}[Proof of Theorem \ref{CohomComplFibKZm}]  
First, we consider $V \backslash \cE$.  For $v_0 \in V \backslash \cE$, the 
map $\varphi : G \to V \backslash \cE$ sending $g \mapsto g\cdot v_0$ is a 
regular covering space map.  Hence, the homomorphism $\varphi_* : 
\pi_1(G) \to \pi_1(V \backslash \cE)$ is injective.  By the assumption on $V 
\backslash \cE$ and the fact that $G$ is homotopy equivalent to its maximal 
torus, both are $\Z^k$, where $k$ is the rank of $G$.  Thus, by the Hurewicz 
theorem and the universal coefficient theorem, we conclude that $\varphi_* 
: H_1(G, \C) \to H_1(V \backslash \cE, \C)$ is injective, and both groups are 
isomorphic to $\C^k$; hence, $\varphi_*$ is an isomorphism.  Thus, also  
$\varphi^* :  H^1(V \backslash \cE, \C) \simeq H^1(G, \C)$.  Hence, if $\{ 
f_1, \dots , f_m\}$ denote the set of basic relative invariants, we shall show 
that $m = k$ and that $\varphi^*(\gw_i)$ for $i = 1, \dots , k$ form a set 
of generators for $H^1(G, \C)$.  
\par
Consider one $f_i$ with its corresponding character $\chi_i$.  Consider a 
one-parameter subgroup $\exp(tw)$ for $w \in \t$, the Lie algebra of a 
maximal torus $T$  of $G$.  Then, for any $v \in V \backslash \cE$,
\begin{equation}
\label{Eqn3.1}
 f_i(\exp(tw)\cdot v) \quad = \quad \chi_i(\exp(tw))\, f_i(v) \, . 
\end{equation}
Since $\exp : \t \to T$ is a Lie group homomorphism, so is $\chi_i \circ 
\exp$.  Thus, if $\{ w_1, \dots , w_k\}$ is a basis for $\t$, then $\chi_i$ 
has the following form on $T$, 
\begin{equation}
\label{Eqn3.2}
\chi_i\left(\exp(t (\sum z_{\ell}\, w_{\ell}))\right) \,\, = \,\, \exp (t (\sum 
\gl_{\ell}^{(i)}\, z_{\ell}))  \, .
\end{equation}
Then, for $w = \sum z_{\ell}\, w_{\ell}$, substituting (\ref{Eqn3.2}) into 
(\ref{Eqn3.1}), and differentiating with respect to $t$, we obtain
\begin{equation}
\label{Eqn3.3}
\pd{f_i(\exp(tw)\cdot v)}{t} \quad = \quad \pd{\exp(t (\sum 
\gl_{\ell}^{(i)}\, z_{\ell}))}{t}\, f_i(v) \, .
\end{equation}
The LHS of (\ref{Eqn3.3}) computes $df_i( \xi_w (\exp(tw)\cdot v))$, where 
$\xi_w$ is the representation vector field associated to $w$.  Thus, we 
obtain
\begin{equation}
\label{Eqn3.4}
df_i( \xi_w (\exp(tw)\cdot v))\quad = \quad (\sum \gl_{\ell}^{(i)}\, z_{\ell}) 
\cdot \exp(t (\sum \gl_{\ell}^{(i)}\, z_{\ell}))\, f_i(v)
\end{equation}
or  (\ref{Eqn3.4}) can be rewritten 
\begin{equation}
\label{Eqn3.5}
\frac{1}{f_i}\cdot df_i( \xi_w) (\exp(tw)\cdot v) \quad = \quad \sum 
\gl_{\ell}^{(i)}\, z_{\ell} \, .
\end{equation}
Hence, 
\begin{equation}
\label{Eqn3.6}
\gw_i(\xi_w) (\exp(tw)\cdot v)) \quad = \quad \sum_{\ell = 1}^{k} 
\gl_{\ell}^{(i)}\, z_{\ell} \, .
\end{equation}
\par
By the Lie--Kolchin theorem, we may suppose that $G$ is a subgroup of a 
Borel subgroup $B_r$ of some $GL_r(\C)$, and the maximal torus $T$ is a 
subgroup of the torus $T^r = (\C^*)^r$.   Thus, we may choose our 
generators $w_j = 2\pi \iti\, u_j$ with $u_j \in \C^r$ so that the $\gg_j(t) = 
\exp(t\, w_j) = \exp (2\pi \iti\, t \, u_j)$, $0 \leq t \leq 1$, each 
parametrizes an $S^1 \subset T$; and the corresponding set of fundamental 
classes for  $j = 1, \dots , k = \rank(G)$,  generate $H_1(T, \Z)$.  Since $T 
\hookrightarrow G$ is a homotopy equivalence, they also generate $H_1(G, 
\Z)$.  Furthermore, their images in $H_1(G, \C)$ form a set of generators 
which are mapped by $\varphi_*$  to a set of generators for 
$H_1(V\backslash \cE, \C)$.  These are defined by $\gd_i(t) = \gg_i(t)\cdot 
v_0$.  \par
Next, we evaluate $\gw_j$ on them.  
\begin{equation}
\label{Eqn3.7}
\int_{\gd_j} \, \gw_i  \,\, =  \,\,  \int_{0}^{1} 
\gw_i(\gd_j^{\prime})(\gg_j(t)\cdot v_0) \, dt \,\, =  \,\, \int_{0}^{1} 
\gw_i(\xi_{w_j})(\exp (t\, w_j)\cdot v_0) \, dt \, .
\end{equation}
Applying (\ref{Eqn3.6}), keeping in mind that for $w_j$, $z_{\ell} = 0$ for 
$\ell \neq j$, we obtain
$$  \int_{0}^{1} \sum \gl_{\ell}^{(i)}\, z_{\ell} \, dt   \,\, = \,\,   \gl_j^{(i)} 
\, . $$
Hence,
\begin{equation}
\label{Eqn3.8}
\int_{\gd_j} \, \gw_i  \,\, =  \,\,  \gl_j^{(i)} \, .
\end{equation}
As we vary over the set of basic relative invariants $\{f_1, \dots , f_m\}$, 
we obtain an $m \times k$ matrix $\gL = (\gl_j^{(i)})$  which  by 
(\ref{Eqn3.2}) yields for the characters $\{ \chi_i \circ \exp : i = 1, \dots , 
m\}$, a representation of the set of corresponding infinitesimal characters 
on $\t$ with respect to the dual basis for $\{ w_1, \dots , w_k\}$.  
\par  

First, by the theory of prehomogeneous vector spaces, \cite[Theorem 
2.9]{Ki}, the set of characters for the basic relative invariants are 
multiplicatively independent in the character group $X(G) \simeq X(G/[G, G]) 
\simeq X(T)$, for $T$ a maximal torus.  This is a free abelian group of rank 
$k = \rank(G) = \rank(T)$.  Hence, $m \leq k$.  
\par  
Second, by \cite[Proposition 2.12]{Ki}, the characters $\{ \chi_i : i = 1, 
\dots , m\}$ generate $X(G_1) \simeq X(T/H)$, where in our case, $G_1$ is 
the quotient of $G$ by the group generated by the unipotent radical $N$ of 
$G$ and the isotropy subgroup of an element $v_0$ in the open orbit.  Here 
$H$ denotes the image of the isotropy subgroup in $G/N \simeq T$.  As a 
consequence of $G$ being solvable, there is a torus $T$ in $G$ so 
composition with projection onto $G/N$ is an isomorphism. Hence, via this 
isomorphism, we may assume $H \subset T$.  As $H$ is finite, $T/H$ is a 
torus of the same dimension and the map $T \to T/H$ induces an 
isomorphism on the corresponding Lie algebras.  Thus, $\{ \chi_i : i = 1, 
\dots , m\}$ generate $X(T/H)$, an abelian group of rank $k$, so $m \geq k$.  
\par
Hence, $m = k$ and the $\{ \chi_i\} $ are algebraically independent in 
$X(T/H)$, which implies the corresponding infinitesimal characters on $\t$ 
are linearly independent.  This is equivalent to $\gL$ being nonsingular.  \par
Hence, by (\ref{Eqn3.8}), we conclude that the $\{ \gw_i\}$ form a set of 
generators for $H^1(V\backslash \cE, \C)$.  \par
Lastly, it remains to show that if $F$ is the Milnor fiber of $(\cE, 0)$, then 
the $\{ \tilde \gw_i\}$ form a spanning set for $H^1(F, \C)$ with single 
relation $\sum_{i = 1}^{k} \tilde \gw_i  = 0$.   By our earlier arguments, if 
$h$ is a reduced defining equation for $\cE$, we may use $F = h^{-1}(1)$.  By 
assumption if $f_i$, $i = 1, \dots , k$ are the basic relative invariants, then 
$h = \prod_{i =1}^{k} f_i$ is a reduced defining equation for $\cE$.  
\par  
We let $\iti : F \hookrightarrow V \backslash \cE$ denote the inclusion, so 
$\tilde \gw_i = \iti^*(\gw_i)$.  As $\iti_* : \pi_1(F) \to \pi_1(V \backslash 
\cE)$ is the inclusion $\Z^{k-1} \hookrightarrow \Z^{k}$ where $k = $ rank 
of $G$, by the Hurewicz theorem and universal coefficient theorem, $\iti^* : 
H^1(V \backslash \cE, \C) \to H^1(F, \C)$ is a surjective map $\C^{k} \to 
\C^{k-1}$.  We need only identify the one-dimensional kernel.  Since $h = 
\prod_{i =1}^{k} f_i$ and $F$ is defined by $h = 1$, we can differentiate the 
equation on $F$ to obtain
\begin{equation}
\label{Eqn3.9}
\sum_{i = 1}^{k} \left. \frac{df_i}{f_i}\right|_{F} \quad = \quad 0 \, ,
\end{equation} 
i.e.  
$$\sum_{i = 1}^{k} \tilde \gw_i\quad = \quad\sum_{i = 1}^{k} 
\iti^*\gw_i\quad = \quad 0 \, . $$
As this is a one-dimensional subspace of $H^1(V \backslash \cE, \C)$ in the 
kernel of $\iti^*$, it must span the entire kernel as claimed.  Since we know 
$F$ is homotopy equivalent to a $(k-1)$--torus, $H^1(F, \C)$ is an exterior 
algebra on $k-1$ generators and these may be chosen to be any $k-1$ of the 
$\{ \tilde \gw_i\}$.
\end{proof}
\section{ Cholesky--Type Factorizations for Parametrized Families}
\label{Sec:5}
\par
We point out a simple consequence of the theorems for the question of 
when, for a continuous or smooth family of complex matrices, 
(modified) Cholesky--type factorization can be continuously or smoothly 
applied to the family of matrices.  To consider all cases together, we view 
(modified) Cholesky--type factorization as giving a factorization $A = B\cdot 
K\cdot C$, for appropriate $K$, $B$ and $C$, with possible relations between 
$B$ and $C$.  For example, for complex symmetric matrices, $K = I$, and 
$B$ is lower triangular with $C = B^T$.  \par 
In each case, for a continuous or smooth family $A_s, s \in X$, we seek 
continuous (or smooth) families $B_s$ and $C_s$ so that $A_s = B_s\cdot 
K\cdot C_s$ for all $s \in X$.  While it may be possible for each individual 
$A_s$ to have a (modified) Cholesky--type factorization, it may not be 
possible to do so in a continuous or smooth manner.  \par
\subsection*{Parametrized Families of Real Matrices}  \hfill 
\par
First for all three real cases of Cholesky-type factorization we have a unique 
representation.  Hence, the orbit map in the real case is a diffeomorphism, 
so we may obtain a continuous or smooth factorization for the family by 
composing with the inverse.  
\par 
However, in this case the open orbit has a much simpler structure.  The real 
solvable groups are not connected, so the open orbits are a union of 
connected components, each of which is diffeomorphic to the connected 
component of the group.  The groups have connected components which have 
as a maximal torus a \lq\lq split torus\rq\rq which is isomorphic to 
$(\R_+)^k \simeq \R^k$, for appropriate $k$.  As the connected component 
is again an extension of this torus by a real nilpotent group, which is again 
Euclidean, we conclude that the connected components are diffeomorphic to 
a Euclidean space, and hence contractible, and the orbit map is a 
diffeomorphism on each component.  
\par  
 Hence, in addition to the continuity or smoothness of factorizations in 
families more is true.  If $(X, Y)$ is a CW--pair and there is a continuous or 
smooth family $A_s, s \in Y$ with Cholesky-type factorization for a given 
type, the $A_s$ can be extended to a continuous or smooth family on $X$ 
which still has continuous, respectively smooth  Cholesky factorization of 
the same type.  

\subsection*{Parametrized Families of Complex Matrices}  \hfill 
\par
By contrast with the real case, as a result of the structure of the 
complement to the exceptional orbit varieties for both Cholesky and modified 
Cholesky--type factorizations, the answer is different.  
\par
First, in the case of general $m \times m$ or $(m-1) \times m$ matrices, 
the LU or modified LU-factorizations are unique.  Hence, the orbit maps $G 
\to V \backslash \cE$ are diffeomorphisms.  Hence, families can be 
continuously or smoothly factored.  However, for both symmetric and skew-
symmetric matrices, there is finite isotropy so the orbit map $G \to V 
\backslash \cE$ is a covering space and the continuous or smooth 
factorization involves lifting a map into $V \backslash \cE$ up to $G$.  There 
are well-known criteria for such a lifting from covering space theory.  We 
show that these conditions can be restated in terms of the cohomology 
classes $\gw_i$ (for each case), so that they define obstructions to such a 
lifting for either symmetric or skew--symmetric Cholesky--type 
factorization.  
\par
First we can use the $\gw_j$ to define integral cohomology classes.  If $\gg$ 
is a smooth closed loop in $V \backslash \cE$, then since $\gw_j = 
f_j^*(\frac{dz}{z})$
$$  \int_{\gg} \, \gw_j \,\, = \,\,  \int_{f_j \circ \gg} \, \frac{dz}{z} \,\, = 
\,\,  2\pi\, \iti\, n  $$
where $n$ is the winding number of $f_j \circ \gg$ about $0$.  Thus, the 
integral of $\frac{1}{2\pi\, \iti}\, \gw_j$ over any smooth closed loop in $V 
\backslash \cE$ is an integer, which implies that it defines an integral 
cohomology class in $H^1(V \backslash \cE, \Z)$.  
In either the symmetric or skew--symmetric case, we can define a 
homomorphism 
\begin{align}
\label{obstrmap}
\bgw : H_1(V \backslash \cE, \Z) &\to  \Z^k \\
u &\mapsto (<\frac{1}{2\pi\, \iti}\, \gw_1, u>, \dots , <\frac{1}{2\pi\, 
\iti}\, \gw_k, u>)  \, . \notag
\end{align}
where $< \cdot , \cdot >$ is the Kronecker product.  Then, we let $\bgw_2$ 
be the composition of $\bgw$ with the projection $\Z^k \to (\Z/2\Z)^k$.
  As a consequence of the results in the preceding sections, we have the 
following corollary.
\begin{Corollary}
\label{Cor3.2}
Suppose $X$ is a locally path--connected space and that $\varphi : X \to M$ 
defines a continuous, respectively smooth (with $X$ a manifold), mapping to 
a space of matrices so that for each $s \in X$, $A_s = \varphi(s)$ has a 
(modified) Cholesky--type factorization for any fixed one of the types 
considered in \S \ref{S:sec1}.  If we are either 
\begin{enumerate}
\item  in the case of general matrices with either Cholesky or modified 
Cholesky--type factorization; or
\item in the symmetric or skew--symmetric cases, and the obstruction 
$\bgw_2 \circ \varphi_* = 0\, ,$ 
\end{enumerate}
then there is a continuous, respectively smooth, (modified) 
Cholesky--type factorization $A_s = B_s\cdot K\cdot C_s$ defined for all $s 
\in X$.
\end{Corollary}
\begin{proof}
By an earlier remark, the conclusion already follows for (1).  It is enough to 
consider case (2).  \par
We may denote the group of matrices acting on $M$ by $G$, and let $\cU$ 
denote the open orbit, which is complement of the exceptional orbit variety 
$\cE$.  Then, by assumption, $\varphi : X \to \cU$.  By composing $\varphi$ 
with an element of $G$, we may suppose $\varphi(s_0) = K$.  Also, we may 
consider each path component of $X$ separately, so we may as well assume 
$X$ is path connected.  \par
By Theorem \ref{CholFacExam} and the proof of Theorem \ref{BlkKpione}, 
for each type of Cholesky--type factorization, the map $p: G \to \cU$ 
sending $g \mapsto g\cdot v_0$ is a smooth finite covering space (where we 
let $v_0 = K$).  Furthermore, by the proof of Theorem \ref{CholFacExam} 
for the symmetric or skew--symmetric cases, $p_* : \pi_1(G, 1) \to 
\pi_1(\cU, v_0)$  is the inclusion $\Z^k \hookrightarrow \Z^k$ with image 
$(2\Z)^k$.  Since $X$ is locally path-connected and 
path-connected, by covering space theory, there is a lift of $\varphi$ to 
$\tilde \varphi : X \to G$ (smooth if $\varphi$ is smooth) with $\tilde 
\varphi(s_0) = 1$, if and only if $\varphi_*(\pi_1(X, s_0)) \subset 
p_*(\pi_1(G, 1))$.  As $\pi_1(\cU, v_0) \simeq \Z^k$ is abelian, by the 
Hurewicz theorem, this is equivalent to $\varphi_*(H_1(X, s_0)) \subset 
(2\Z)^k$.  However, this holds exactly when $\bgw_2 \circ \varphi_* = 0$ 
\par
Then, by the definition of the covering map $p$, the lift gives the continuous, 
resp. smooth, Cholesky-type factorization for all $s \in X$.
\end{proof}
\begin{Remark}
The obstruction in Corollary \ref{Cor3.2}, will always vanish if e.g. $H_1(X, 
\Z)$ is a torsion group. \par
 From the structure of the complement $V \backslash \cE$ being homotopy 
equivalent to a torus for any of the cases of Cholesky or modified Cholesky--
type factorization, we can also give a sufficient condition for the extension 
problem.  For a CW-pair $(X, Y)$, a sufficient condition for the extension of a 
continuous or smooth family $A_s$ on $Y$, which has a continuous or 
smooth Cholesky factorization of given type, to a continuous or smooth 
family on $X$ having the Cholesky factorization of the same type is that $(X, 
Y)$ is $1$--connected.
\end{Remark}
 
\end{document}